\documentclass{article}
\usepackage[utf8]{inputenc}
\usepackage[T2A]{fontenc}
\usepackage[english]{babel}

\usepackage{parskip}

\usepackage{amsmath}
\usepackage{amssymb}
\usepackage{amsthm}
\usepackage{amsfonts}
\usepackage{asymptote}
\usepackage{xcolor}
\usepackage{subcaption}
\usepackage{graphicx}
\usepackage{cancel}

\newtheorem{proposition}{Proposition}[]

\newtheorem{theorem}{Theorem}[]

\theoremstyle{definition}
\newtheorem{definition}{Definition}[]
\newtheorem{example}{Example}[]
\theoremstyle{plain}

\newcommand{\diam}{\operatorname{diam}}

\newcommand{\dis}{\operatorname{dis}}
\newcommand{\dist}{\operatorname{d}}

\renewcommand{\:}{\colon}

\renewcommand{\ss}{\subset}

\newcommand{\N}{\mathbb{N}}
\newcommand{\R}{\mathbb{R}}
\newcommand{\Z}{\mathbb{Z}}

\newcommand{\St}{\operatorname{St}}
\newcommand{\asdim}{\operatorname{asdim}}

\title{Calculating Gromov--Hausdorff distance by means of asymptotic dimension}
\author{Ivan N. Mikhailov}
\date{}

\begin{document}
\maketitle

\begin{abstract}
In this paper, we apply the concept of asymptotic dimension to calculating Gromov--Hausdorff distances between some unbounded metric spaces. In particular, we show that the Gromov--Hausdorff between $\R^2$ with the Euclidean metric and $\Z^2$ equals the Hausdorff distance between them: $d_{GH}(\R^2, \Z^2) = d_H(\R^2, \Z^2)$. 
\end{abstract}



\section{Introduction}

The Gromov--Hausdorff distance is an important notion from metric geometry that allows to compare any two abstract metric spaces to each other \cite{BBI}, \cite{TuzhilinHistory}. Very few exact values of Gromov--Hausdorff distances are known (see, for example, \cite{LMZ}, \cite{SpheresNew}, \cite{SegmentCircle}). In this paper we show how the concept of asymptotic dimension of a metric space can be applied to calculating Gromov--Hausdorff distance. For example, we show that $d_{GH}(\R^2, \Z^2) = d_H(\R^2, \Z^2)$.

\section{Preliminaries}

\emph{A metric space} is an arbitrary pair $(X,\,\dist_X)$, where $X$ is an arbitrary set, $\dist_X\: X\times X\to [0,\,\infty)$ is some metric on it, that is, a nonnegative symmetric, positively definite function that satisfies the triangle inequality.

For convenience, if it is clear in which metric space we are working, we denote the distance between points $x$ and $y$ by $|xy|$. Suppose $X$ is a metric space. By $U_r(a) =\{x\in X\colon |ax|<r\}$, $B_r(a) = \{x\in X\colon |ax|\le r\}$ we denote open and closed balls centered at point~$a$ of radius~$r$ in~$X$. For an arbitrary subset $A\subset X$ of a metric space~$X$, let $U_r(A) = \cup_{a\in A} U_r(a)$ be an open $r$-neighborhood of~$A$. For non-empty subsets $A \ss X$, $B \ss X$ we put $\dist(A,\,B)=\inf\bigl\{|ab|:\,a\in A,\,b\in B\bigl\}$.

\subsection{Hausdorff and Gromov--Hausdorff distances}

\begin{definition}
Let $A$ and $B$ be non-empty subsets of a metric space.
\emph{The Hausdorff distance} between $A$ and~$B$ is the quantity $$\dist_H(A,\,B) = \inf\bigl\{r > 0\colon A\subset U_r(B),\,B\subset U_r(A)\bigr\}.$$
\end{definition}
\begin{definition} Let $X$ and $Y$ be metric spaces. The triple $(X', Y', Z)$, consisting of a metric space $Z$ and its two subsets $X'$ and $Y'$, isometric to $X$ and~$Y$ respectively, is called \emph{a realization of the pair} $(X, Y)$.
\end{definition}

\begin{definition} \emph{The Gromov-Hausdorff distance} $d_{GH} (X, Y)$ between $X$ and~$Y$ is the exact lower bound of the numbers $r\ge 0$ for which there exists a realization $(X', Y', Z)$ of the pair $(X, Y)$ such that $\dist_H(X',\,Y') \le r$. 
\end{definition}


Now let $X,\,Y$ be non-empty sets.  

\begin{definition} Each $\sigma\subset X\times Y$ is called a \textit{relation} between $X$ and~$Y$.
\end{definition}

By $\mathcal{P}_0(X,\,Y)$ we denote the set of all non-empty relations between $X$ and~$Y$.

We put $$\pi_X\colon X\times Y\rightarrow X,\;\pi_X(x,\,y) = x,$$ $$\pi_Y\colon X\times Y\rightarrow Y,\;\pi_Y(x,\,y) = y.$$ 

\begin{definition} A relation $R\subset X\times Y$ is called a \textit{correspondence}, if restrictions $\pi_X|_R$ and $\pi_Y|_R$ are surjective.
\end{definition}

Let $\mathcal{R}(X,\,Y)$ be the set of all correspondences between $X$ and~$Y$.

\begin{definition} Let $X,\,Y$ be metric spaces, $\sigma \in \mathcal{P}_0(X,\,Y)$. The \textit{distortion} of $\sigma$ is the quantity $$\dis \sigma = \sup\Bigl\{\bigl||xx'|-|yy'|\bigr|\colon(x,\,y),\,(x',\,y')\in\sigma\Bigr\}.$$
\end{definition}

\begin{proposition}[\cite{BBI}]  \label{proposition: distGHformula}
For arbitrary metric spaces $X$ and~$Y$, the following equality holds $$2\dist_{GH}(X,\,Y) = \inf\bigl\{\dis\,R\colon R\in\mathcal{R}(X,\,Y)\bigr\}.$$
\end{proposition}

Denote the set of all positive real numbers by $\R_{>0}$.

\begin{definition}
Following \cite{Stabilizers}, we define a \emph{stabilizer} of a metric space~$X$ as follows: $$\St X = \bigl\{\lambda \in\R_{>0}\: \lambda X\;\text{is isometric to}\;X\bigr\}.$$
\end{definition}

It is not difficult to show that $\St X$ is a subgroup in a multiplicative group $\R_{>0}$.

\subsection{Asymptotic dimension}

Let $X$ be an arbitrary metric space~$X$. 

\begin{definition}
It is said that the asymptotic dimension of~$X$ does not exceed~$n$ and written $\asdim X\le n$ iff for every uniformly bounded open cover $\mathcal{V}$ of~$X$ there is a uniformly bounded open cover~$\mathcal{U}$ of~$X$ of multiplicity~$\le(n+1)$ so that $\mathcal{V}$ refines $\mathcal{U}$. Then $\asdim X = n$ iff $\asdim X \le n$ and $\asdim X\, \cancel{\le}\, n-1$.
\end{definition}

We will need the equivalent definition of the asymptotic dimension \cite{asdim}[Theorem 19 (2), p.7].

\begin{definition}
The family $\mathcal{U} = \{U_\alpha\}_\alpha$ of subsets of a metric space~$X$ is said to be $r$-disjoint iff $d(U_\alpha, U_\beta) > r$ for all possible $\alpha\neq\beta$.
\end{definition}

\begin{definition}
We have $\asdim X\le n$ iff, for every $r < \infty$, there exist $r$-disjoint families $\mathcal{U}^0,\ldots \mathcal{U}^n$ of uniformly bounded subsets of~$X$ such that $\cup_i\mathcal{U}^i$ is a cover of~$X$.
\end{definition}

\begin{theorem}[\cite{LargeScaleGeometry}]\label{thm: asdimRn}
The following equality holds: $\asdim \R^n = n$.
\end{theorem}

Since every two norms on a finite-dimensional linear space are equivalent to each other, we obtain the following enhancement:

\begin{theorem}\label{thm: asdimNorm}
For an arbitrary $n$-dimensional normed space $X$, we have \mbox{$\asdim X = n$.}
\end{theorem}

\clearpage

\section{Main theorem}

\begin{theorem}\label{thm: asdimGH}
Let $X$ be a metric space with $\asdim X\ge n$ and $\St X\neq \{e\}$. Suppose in a metric space $A$ there exist $r$-disjoint families $\mathcal{U}^1,\ldots, \mathcal{U}^k$, $1\le k\le n$ of uniformly bounded subsets of $A$ such that $A\ss\cup_i\mathcal{U}^i$. Then $d_{GH}(A, X)\ge \frac{r}{2}$.
\end{theorem}

\begin{proof}
Consider an arbitrary correspondence $R\in\mathcal{R}(A, X)$. Suppose that $\dis R < r$. Then there exists some $\varepsilon > 0$ such that $\dis R < r-\varepsilon$. 

Let $\mathcal{U}^i = \{U_\alpha^i\}_\alpha$. We put $V^i_\alpha = R(U^i_\alpha)$ and $\mathcal{V}^i =  \{V^i_\alpha\}_\alpha$. 

Since the families $\mathcal{U}^i$, $i = 1,\ldots, k$ consist of uniformly bounded subsets of $A$, there exists $C > 0$ such that $\diam U^i_\alpha < C$ for arbitrary $i$, $\alpha$. Since $\dis R < r -\varepsilon$, we obtain $\diam V^i_\alpha < C+ r - \varepsilon$. Therefore, the families $\mathcal{V}^i$ consist of uniformly bounded subsets of $X$. 

By condition, for arbitrary possible $\alpha\neq \beta$ and arbitrary~$i$, we have $d(U^{i}_\alpha, U^i_\beta) > r$. Since $\dis R < r-\varepsilon$, we obtain $d(V^i_\alpha, V^i_\beta)> \varepsilon$. Therefore, the families $\mathcal{V}^1,\ldots, \mathcal{V}^k$ are $\varepsilon$-disjoint in $X$. 

Since $R$ is a correspondence and $A = \cup_i \mathcal{U}^i$, we obtain $X = \cup_i \mathcal{V}^i$.

Since $\St X \neq \{e\}$, there exists $\lambda\in \St X$ such that $\lambda > 1$. Consider the families $\lambda^n\mathcal{V}^1,\ldots \lambda^n\mathcal{V}^k$ of subsets of $X$ for $n\in \N$. These families are $(\lambda^n\varepsilon)$-separated and consist of subsets whose diameters are bounded by $\lambda^n(C+r-\varepsilon)$ from above. 

Now we observe that due to the arbitrariness of $n\in\N$ we have obtained $$\asdim X \le k-1 < n $$ by the definition of the asymptotic dimension. This is a contradiction.
\end{proof}

\section{Various $\varepsilon$-nets in $\Z^2$}

\begin{example}

We show that $d_H(\Z^2, \R^2) = d_{GH}(\Z^2,\R^2)$ for $\R^2$ with the Euclidean metric. 

Since $d_{GH}(\Z^2,\R^2)\le d_H(\Z^2,\R^2) = \frac{\sqrt{2}}{2}$, it suffices to prove that for an arbitrary correspondence $R\in\mathcal{R}(\Z^2,\,\R^2)$ the inequality $\dis R\ge\sqrt{2}$ holds. 

Consider chess colouring of $\Z^2$ (see Figure~\ref{fig: ChessZ}).

\begin{figure}[h]
    \centering
    \includegraphics[width=0.5\linewidth]{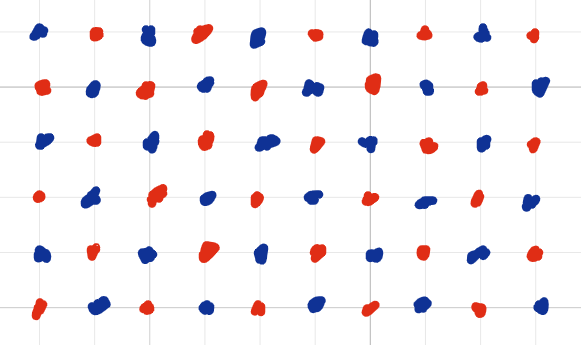}
    \caption{}
    \label{fig: ChessZ}
\end{figure}

Note that the families of red points and blue points are both $\sqrt{2}$-disjoint. Thus, by Theorem~\ref{thm: asdimGH}, we obtain that $d_{GH}(\Z^2, \R^2)\ge d_H(\Z^2, \R^2) = \frac{\sqrt{2}}{2}$. Hence, $d_{GH}(\Z^2, \R^2)= d_H(\Z^2, \R^2)$.
    
\end{example}

\begin{example}

Consider the following subset $A$ of $\R^2$ with the Euclidean metric $$A = \bigl\{(x, 0)\: x\in\R\bigr\}\cup \cup_{n\in \Z}\bigl\{(n, y)\:y\in\R\bigr\}.$$ It is clear that $d_{GH}(A, \R^2)\le d_H(A, \R^2) = \frac{1}{2}$. 

We can cover $A$ with two families of uniformly bounded subsets such that these families are $1$-disjoint (see Figure~\ref{example2} below). Thus, $d_{GH}(A, \R^2)\ge \frac{1}{2}$ and, hence, $d_{GH}(A, \R^2) = d_H(A, \R^2) = \frac{1}{2}$.

\begin{figure}[h]
    \centering
    \includegraphics[width=0.8\linewidth]{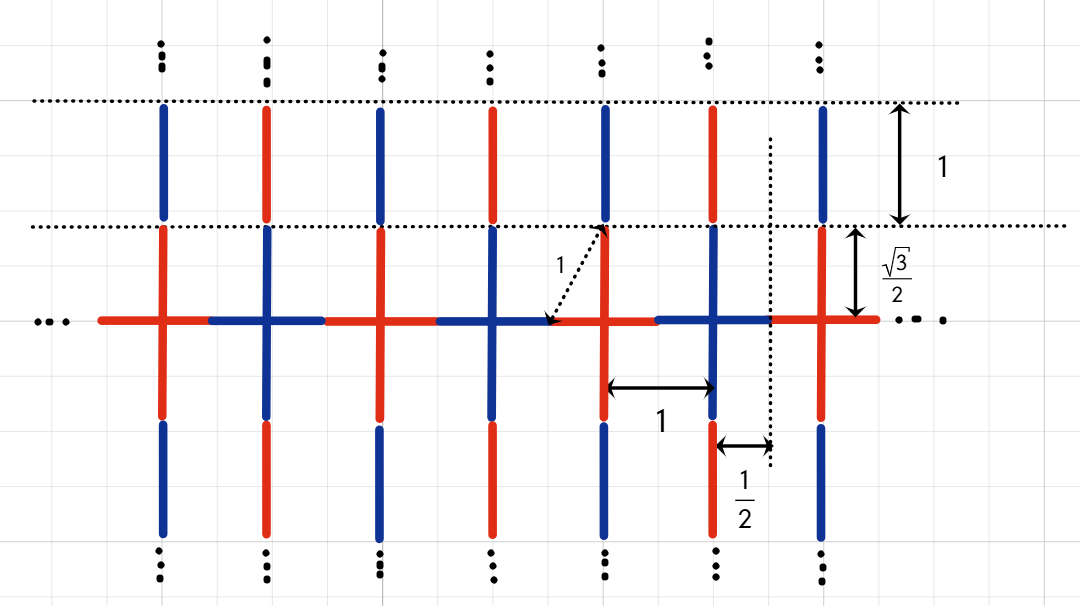}
    \caption{}
    \label{example2}
\end{figure}
    
\end{example}



\begin{thebibliography}{2}
\bibitem{asdim} G.\,Bell, A.\,Dranishnikov, \emph{Asymptotic Dimension}, Topology Appl. 155 (2008),
1265–1296.
\bibitem{BBI} D.\,Burago, Yu.\,Burago, S.\,Ivanov, \emph{A Course in Metric Geometry}, Graduate Studies in Mathematics 33, AMS, 2001.
\bibitem{LargeScaleGeometry} Piotr W.\,Nowak, Guolinag Yu, \emph{Large Scale Geometry}, 2012.
\bibitem{Stabilizers} S. I. Bogataya, S. A. Bogatyy, V. V. Redkozubov, A. A. Tuzhilin, \emph{Clouds in Gromov–Hausdorff Class: their
completeness and centers}, ArXiv e-prints, 	arXiv:2202.07337, 2022.


\bibitem{SpheresNew} Saúl Rodríguez Martín, \emph{Some novel constructions of optimal Gromov-Hausdorff-optimal correspondences between spheres}, 2024, ArXiv e-prints, 	arXiv:2409.02248
\bibitem{LMZ} Sunhyuk Lim, Facundo Memoli, Zane Smith, \emph{The Gromov–Hausdorff distance between
spheres}, 2022, ArXiv e-prints, arXiv:2105.00611v5.
\bibitem{TuzhilinHistory} A.\,A.\,Tuzhilin, \emph{Who invented the Gromov-Hausdorff Distance?}, 2016, ArXiv e-prints, arXiv:1612.00728.
\bibitem{SegmentCircle} Yibo Ji, A.\,A. Tuzhilin, \emph{Gromov-Hausdorff Distance Between Segment and Circle}, 2021, ArXiv e-prints, arXiv:2101.05762.



\end{thebibliography}
\end{document}